\documentclass[reqno,11pt]{amsart}
\usepackage[tmargin=30mm,bmargin=30mm,hmargin=25mm]{geometry}
\usepackage{amsmath}
\usepackage{amssymb}
\usepackage{amsthm}
\usepackage{eucal}
\usepackage{enumerate}
\input xy
\xyoption{all}
\numberwithin{equation}{section}
\usepackage{hyperref}

\newtheorem{theorem}{Theorem}[section]

\newtheorem{lemma}[theorem]{Lemma}
\newtheorem{proposition}[theorem]{Proposition}
\theoremstyle{definition}
\newtheorem{definition}[theorem]{Definition}
\newtheorem{example}[theorem]{Example}

\begin{document}
\title[Quasi-hereditary rings are closed under block extensions]{Quasi-hereditary rings are closed under taking block extensions}

\author{Takahide Adachi}
\address{T.~Adachi: Faculty of Global and Science Studies, Yamaguchi University, 1677-1 Yoshida, Yamaguchi 753-8541, Japan}
\email{tadachi@yamaguchi-u.ac.jp}

\author{Mayu Tsukamoto}
\address{M.~Tsukamoto: Graduate school of Sciences and Technology for Innovation, Yamaguchi University, 1677-1 Yoshida, Yamaguchi 753-8512, Japan}
\email{tsukamot@yamaguchi-u.ac.jp}

\subjclass[2020]{Primary 16S50, Secondly 16L30}
\keywords{quasi-hereditary ring, Morita context ring, block extension}

\begin{abstract}
In this paper, we give a sufficient condition for Morita context rings to be quasi-hereditary.
As an application, we show that each block extension of a quasi-hereditary ring is also quasi-hereditary.
\end{abstract}
\maketitle

\section{Introduction}

Cline, Parshall and Scott introduced the notion of quasi-hereditary algebras to deal the representation theory of complex Lie algebras and algebraic groups. 
Ever since their introduction, quasi-hereditary algebras have extensively been studied from the viewpoint of the representation theory of algebras. 
In particular, the works of Dlab and Ringel gave a ring theoretic approach to quasi-hereditary algebras and generally quasi-hereditary semiprimary rings (\cite{DR891, DR892, DR893, R91}). 

Motivated by Morita theorem on equivalences of module categories, Bass (\cite{B62}) introduced the notion of Morita contexts. 
Recall that a Morita context is a $6$-tuple $\mathcal{M}:=(R,S,M,N,\varphi, \psi)$, where $R,S$ are associative unital rings, $M$ is an $R$-$S$-bimodule, $N$ is an $S$-$R$-bimodule, $\varphi:M\otimes_{S}N\to R$ is an $R$-$R$-bimodule homomorphism and $\psi:N\otimes_{R}M\to S$ is an $S$-$S$-bimodule homomorphism.
A Morita context ring is a $2\times 2$ matrix ring $\Lambda(\mathcal{M}):=\left(\begin{smallmatrix} R&M \\N&S\end{smallmatrix}\right)$ associated to a Morita context $\mathcal{M}$. 
It is known that using Morita context rings, all quasi-hereditary algebras can be constructed from semisimple algebras (\cite{PS88, DR893}). 
In this paper, we give a sufficient condition on a Morita context $\mathcal{M}$ so that $\Lambda(\mathcal{M})$ is quasi-hereditary. 

\begin{theorem}[{Theorem \ref{thm:main2}}]\label{thm:intro-thm1}
Let $\Lambda(\mathcal{M})$ be a Morita context ring which is semiprimary.
Assume that the following four conditions are satisfied.
\begin{itemize}
\item[(a)] $R/\operatorname{Im}\varphi$ and $S$ are quasi-hereditary rings.
\item[(b)] $M$ is a projective $S$-module.
\item[(c)] $N$ is a projective $S^{\mathrm{op}}$-module.
\item[(d)] $\varphi$ is a monomorphism.
\end{itemize}
Then $\Lambda(\mathcal{M})$ is a quasi-hereditary ring.
\end{theorem}

As an application of Theorem \ref{thm:intro-thm1}, we study block extensions (see Definition \ref{def:block-ext}) of quasi-hereditary rings.
The notion of block extensions of semiperfect rings was introduced by Oshiro (\cite{O90}) in a structure theorem of Harada rings.
As a modern representation theoretic approach, Yamaura (\cite{GY10}) gave a description of quivers with relations of block extensions of finite dimensional algebras. 
Since a block extension is a certain matrix ring, it is realized as a Morita context ring.
Using Theorem \ref{thm:intro-thm1}, we prove that each block extension of a quasi-hereditary ring is also quasi-hereditary. 
 
\begin{theorem}[Theorem \ref{thm:main3}]
Let $R$ be a basic semiprimary ring with complete set $\{ e_{1}, e_{2}, \ldots, e_{n}\}$ of primitive orthogonal idempotents of $R$ and let $\ell_{1}, \ell_{2}, \ldots, \ell_{n}$ be positive integers. 
If $R$ is quasi-hereditary, then so is the block extension $B(R;\ell_{1}, \ell_{2},\ldots,\ell_{n})$.
\end{theorem}

\subsection*{Notation and convention}

Throughout this paper, $R$ is an associative unital ring and $J(R)$ is the Jacobson radical of $R$.
Let $R^{\mathrm{op}}$ denote the opposite ring of $R$.
By a module, we mean a right module unless otherwise stated.
Let $\mathsf{Mod}R$ denote the category of all $R$-modules, $\mathsf{Proj}R$ the category of all projective $R$-modules and $\mathsf{Add}M$ the category of all $R$-modules isomorphic to a direct summand of an arbitrary direct sum of copies of an $R$-module $M$.
For two integers $i\leq j$, put $[i,j]:=\{ k \in \mathbb{Z}\mid i\leq k\leq j\}$.

\section{Quasi-hereditary rings}

In this section, we recall the definition of quasi-hereditary (semiprimary) rings and give a sufficient condition for a semiprimary ring to be quasi-hereditary.
We begin with recalling the definition of semiprimary rings. 
For details, see \cite{AF91}. 

\begin{definition}
A ring $R$ is said to be \emph{semiprimary} if the Jacobson radical $J(R)$ is a nilpotent ideal and $R/J(R)$ is a semisimple (artin) ring.
\end{definition}

Note that, if $R$ is a semiprimary ring and $e$ is an idempotent of $R$, then $R/ReR$ and $eRe$ are also semiprimary (for example, see \cite[\S 27 and \S 28]{AF91}).

To define quasi-hereditary rings, we recall the notion of idempotent ideals and heredity ideals.
Let $R$ be an arbitrary ring. 
A two-sided ideal $I$ of $R$ is called an \emph{idempotent ideal} if $I^{2}=I$ holds.
We collect basic properties of idempotent ideals.

\begin{lemma}[{\cite[Statement 6]{DR892}}]\label{lem:DR2-6}
If $e$ is an idempotent of $R$, then $ReR$ is an idempotent ideal.
Conversely, if $R$ is semiprimary and $I$ is an idempotent ideal, then there exists an idempotent $e$ of $R$ such that $I=ReR$.
\end{lemma}

\begin{lemma}\label{lem:lift-idem}
Let $R$ be an arbitrary ring and $I'\subseteq I$ two-sided ideals of $R$.
If $I'\subseteq R$ and $I/I'\subseteq R/I'$ are idempotent ideals, then so is $I\subseteq R$.
\end{lemma}
\begin{proof}
It is enough to show $I\subseteq I^{2}$.
Let $i\in I$. 
By $(I/I')^{2}=I/I'$, there exist $j,k\in I$ such that $i+I'=\sum(j+I')(k+I')$.
Thus $i-\sum jk\in I'=(I')^{2}$.
There exist $l,m\in I'$ such that $i-\sum jk=\sum lm$.
Therefore $i=\sum jk+\sum lm\in I^{2}$. 
\end{proof}

We frequently use the following well-known results (see \cite[Statement 7]{DR892}, \cite[Lemma 2]{DR893} and \cite[Lemma 3.8]{T20}).

\begin{lemma}\label{lem:keyDRT}
Let $R$ be an arbitrary ring and let $e\in R$ be an idempotent with $ReR\in \mathsf{Proj}R$.
Then the following statements hold.
\begin{itemize}
\item[(1)] $ReR\in \mathsf{Add}(eR)$. 
\item[(2)] $Re\in \mathsf{Proj}(eRe)$.
\item[(3)] The multiplication $\mu:Re\otimes_{eRe}eR\to ReR$ is an isomorphism (as an $R$-$R$-bimodule).
\item[(4)] For each idempotent $f\in R$, the multiplication $\nu: ReR\otimes_{R}RfR\rightarrow ReRfR$ is an isomorphism (as an $R$-$R$-bimodule).
\end{itemize}
\end{lemma}

For the convenience of readers, we give a proof.

\begin{proof}
(1) Take an epimorphism $(eRe)^{\oplus}\rightarrow Re$ in $\mathsf{Mod}(eRe)$. Then we have epimorphisms
\begin{align}
eR^{\oplus}\cong (eRe)^{\oplus} \otimes_{eRe}eR\rightarrow Re\otimes_{eRe}eR \rightarrow ReR \notag
\end{align}
in $\mathsf{Mod}R$. 
Since $ReR$ is a projective $R$-module, the composition is a split epimorphism.
Hence $ReR\in \mathsf{Add}(eR)$.

(2) By (1), we have $Re=ReRe\cong ReR\otimes_{R}Re\in \mathsf{Proj}(eRe)$.

(3) The multiplication $\mu_{eR}: eR\otimes_{R}Re\otimes_{eRe}eR\to eRe\otimes_{eRe}eR\to eR$ is an isomorphism. 
By $ReR\in \mathsf{Add}(eR)$, the multiplication
\begin{align}
\mu_{ReR}: ReR\otimes_{R}Re\otimes_{eRe}eR\rightarrow ReR \notag
\end{align}
is an isomorphism.
On the other hand, since $\varphi: Re\rightarrow ReR\otimes_{R}Re$ $(a\mapsto a\otimes e)$ is an isomorphism, so is $\varphi\otimes eR: Re\otimes_{eRe} eR\rightarrow ReR\otimes_{R}Re\otimes_{eRe}eR$.
By $\mu=\mu_{ReR}\circ(\varphi\otimes eR)$, we have the assertion.

(4) By \cite[Chapter VI, Exercises 19]{CE56}, we have a short exact sequence
\begin{align}
0\rightarrow \operatorname{Tor}^{R}_{2}(R/ReR,R/RfR)\rightarrow ReR\otimes_{R}RfR\xrightarrow{\nu} ReRfR\rightarrow 0. \notag
\end{align}
Applying $-\otimes_{R}R/RfR$ to a short exact sequence 
\begin{align}
0\rightarrow ReR\rightarrow R\rightarrow R/ReR\rightarrow 0 \notag
\end{align}
induces an exact sequence
\begin{align}
0=\operatorname{Tor}^{R}_{2}(R,R/RfR)\rightarrow \operatorname{Tor}^{R}_{2}(R/ReR,R/RfR)\rightarrow \operatorname{Tor}^{R}_{1}(ReR,R/RfR).\notag
\end{align}
By $ReR\in \mathsf{Proj}R$, we obtain $\operatorname{Tor}^{R}_{1}(ReR,R/RfR)=0$.
This implies $\operatorname{Tor}^{R}_{2}(R/ReR,R/RfR)=0$.
Hence $\nu$ is an isomorphism.
\end{proof}

An idempotent ideal $H$ of $R$ is called a \emph{heredity ideal} if $H$ is a projective $R$-module and $HJ(R)H=0$. 
If $H=ReR$ holds for some idempotent $e\in R$, then $HJ(R)H=0$ if and only if $eJ(R)e=0$.
We recall the definition of quasi-hereditary rings.

\begin{definition}[{\cite{S87, CPS88, DR892}}]
Let $R$ be a semiprimary ring.
\begin{itemize}
\item[(1)] A chain of two-sided ideals of $R$
\begin{align}\label{eq:chain-hered}
R=H_{0}\supseteq H_{1}\supseteq \cdots \supseteq H_{m-1}\supseteq H_{m}=0
\end{align}
is called a \emph{heredity chain} of $R$ if $H_{i}/H_{i+1}$ is a heredity ideal of $R/H_{i+1}$ for each $i\in [0, m-1]$.
\item[(2)] $R$ is called a \emph{quasi-hereditary ring} if it admits a heredity chain.
\end{itemize}
\end{definition}

The heredity chain \eqref{eq:chain-hered} has the following properties. 
For each $i\in[0,m-1]$, it follows from Lemma \ref{lem:lift-idem} that $H_{i}$ is an idempotent ideal of $R$.
Moreover, by Lemma \ref{lem:DR2-6}, there exists an idempotent $\varepsilon_{i}\in R$ such that $H_{i}=R\varepsilon_{i}R$.
Using \cite[Lemma 1.1]{UY90}, we can take idempotents $\varepsilon_{0},\varepsilon_{1},\ldots, \varepsilon_{m-1}$ satisfying $\varepsilon_{i}\varepsilon_{j}=\varepsilon_{j}\varepsilon_{i}=\varepsilon_{j}$ whenever $i\leq j$.
In the following, we choose such idempotents unless otherwise stated.

\begin{proposition}[{\cite[Statements 7 and 10]{DR892}}]\label{prop:DR2-710}
Let $R$ be a quasi-hereditary ring, that is, there exists a heredity chain
\begin{align}
R=H_{0} \supseteq H_{1} \supseteq \cdots \supseteq H_{m-1} \supseteq H_{m}=0. \notag
\end{align}
such that $H_{i}=R\varepsilon_{i}R$ for some idempotent $\varepsilon_{i}\in R$.
Then $R/R\varepsilon_{i}R$ and $\varepsilon_{i}R\varepsilon_{i}$ are quasi-hereditary for each $i\in[0,m-1]$, and $H_{m-1}\in \mathsf{Proj}R\cap\mathsf{Proj}R^{\mathrm{op}}$. 
\end{proposition}

We give a criterion for a semiprimary ring to be quasi-hereditary.

\begin{theorem}\label{thm:main1}
Let $R$ be a semiprimary ring.
Assume that an idempotent $e\in R$ satisfies the following three conditions.
\begin{itemize}
\item[(a)] $R/ReR$ and $eRe$ are quasi-hereditary rings.
\item[(b)] $ReR$ is a projective $R$-module.
\item[(c)] $ReR$ is a projective $R^{\mathrm{op}}$-module.
\end{itemize}
Then $R$ is a quasi-hereditary ring.
\end{theorem}

In the rest of this section, we show Theorem \ref{thm:main1}.
The following lemma plays an important role.

\begin{lemma}\label{lem:key-idem}
Let $R$ be an arbitrary ring and let $e, f\in R$ be idempotents with $ef=fe=f$. 
Assume $ReR\in \mathsf{Proj}R\cap\mathsf{Proj}R^{\mathrm{op}}$ and $eRfRe\in \mathsf{Proj}(eRe)$.
Then $RfR\in \mathsf{Proj}R$.
\end{lemma}

\begin{proof}
By Lemma \ref{lem:keyDRT}(2), the assumption $ReR\in\mathsf{Proj}R$ gives $Re\in \mathsf{Proj}(eRe)$.
Thus we have $Re\otimes_{eRe}eRfRe\in \mathsf{Proj}(eRe)$ by $eRfRe\in \mathsf{Proj}(eRe)$.
Since $-\otimes_{eRe}eR$ preserves projectivity, we obtain $Re\otimes_{eRe}eRfRe\otimes_{eRe}eR\in \mathsf{Proj}R$.
By Lemma \ref{lem:keyDRT}(3), we have
\begin{align}
Re\otimes_{eRe}eRfRe\otimes_{eRe}eR
&\cong Re\otimes_{eRe}eR\otimes_{R}RfR\otimes_{R}Re\otimes_{eRe}eR\notag\\
&\cong ReR\otimes_{R}RfR\otimes_{R}ReR.\notag
\end{align}
Moreover, the assumption $ReR\in\mathsf{Proj}R\cap\mathsf{Proj}R^{\mathrm{op}}$ induces
\begin{align}
ReR\otimes_{R}RfR\otimes_{R}ReR
&\cong ReRfR\otimes_{R}ReR =RfR\otimes_{R}ReR\notag\\
&\cong RfReR=RfR\notag
\end{align}
by Lemma \ref{lem:keyDRT}(4) and its dual statement.
Hence $RfR\in\mathsf{Proj}R$.
\end{proof}

If $R/ReR$ and $eRe$ are quasi-hereditary for some idempotent $e\in R$, then we have a chain of idempotent ideals of $R$ as follows.

\begin{proposition}\label{prop:key-qhs}
Let $R$ be a semiprimary ring and let $e$ be an idempotent of $R$.
Then the following statements hold.
\begin{itemize}
\item[(1)] If $R/ReR$ is a quasi-hereditary ring, then there exists a chain of two-sided ideals of $R$
\begin{align}
R= R\varepsilon_{0}R \supseteq R\varepsilon_{1}R \supseteq \cdots \supseteq R\varepsilon_{l-1}R \supseteq R\varepsilon_{l}R=ReR\notag
\end{align}
such that for each $i\in [0,l-1]$, $R\varepsilon_{i}R/R\varepsilon_{i+1}R$ is a heredity ideal of $R/R\varepsilon_{i+1}R$.
\item[(2)] If $eRe$ is a quasi-hereditary ring and $ReR\in \mathsf{Proj}R\cap \mathsf{Proj}R^{\mathrm{op}}$, then there exists a chain of two-sided ideals of $R$ 
\begin{align}
ReR=R\varepsilon_{0}R \supseteq R\varepsilon_{1}R \supseteq \cdots \supseteq R\varepsilon_{m-1}R \supseteq R\varepsilon_{m}R=0\notag
\end{align}
such that for each $j\in [0,m-1]$, $R\varepsilon_{j}R/R\varepsilon_{j+1}R$ is a heredity ideal of $R/R\varepsilon_{j+1}R$.
\end{itemize}
\end{proposition}

\begin{proof}
(1) Since $R/ReR$ is quasi-hereditary, we have a heredity chain
\begin{align}
R/ReR=H_{0}/ReR \supseteq H_{1}/ReR \supseteq \cdots \supseteq H_{l-1}/ReR \supseteq H_{l}/ReR=0. \notag
\end{align}
By Lemmas \ref{lem:DR2-6} and \ref{lem:lift-idem}, there exists an idempotent $\varepsilon_{i}\in R$ such that $H_{i}=R\varepsilon_{i}R$.
Hence we obtain the desired chain of two-sided ideals.

(2) Since $eRe$ is quasi-hereditary, we have a heredity chain 
\begin{align}
eRe=eR\varepsilon_{0}Re\supseteq eR\varepsilon_{1}Re \supseteq \cdots \supseteq eR\varepsilon_{m-1}Re \supseteq eR\varepsilon_{m}Re=0.\notag
\end{align}
Fix $j\in[0,m-1]$ and let $\overline{(-)}: R\to R/R\varepsilon_{j+1}R$ be a natural surjection.
Since $\overline{e}\overline{R}\overline{\varepsilon_{j}}\overline{R}\overline{e}$ is a heredity ideal of $\overline{e}\overline{R}\overline{e}$, we obtain $\overline{e}\overline{R}\overline{\varepsilon_{j}}\overline{R}\overline{e}\in \mathsf{Proj}(\overline{e}\overline{R}\overline{e})$ and 
$\overline{\varepsilon_{j}}J(\overline{R})\overline{\varepsilon_{j}}=\overline{\varepsilon_{j}}\overline{e}J(\overline{R})\overline{e}\overline{\varepsilon_{j}}=\overline{\varepsilon_{j}}J(\overline{e}\overline{R}\overline{e})\overline{\varepsilon_{j}}=0$.
Moreover, $ReR\in\mathsf{Proj}R$ and $ReR\in \mathsf{Proj}R^{\mathrm{op}}$ imply $\overline{R}\overline{e}\overline{R}\cong ReR\otimes_{R}\overline{R}\in \mathsf{Proj}\overline{R}$ and $\overline{R}\overline{e}\overline{R}\in \mathsf{Proj}\overline{R}^{\mathrm{op}}$ respectively.
By Lemma \ref{lem:key-idem}, $\overline{R}\overline{\varepsilon}_{j}\overline{R}\in \mathsf{Proj}\overline{R}$.
This completes the proof.
\end{proof}

Now, we are ready to prove Theorem \ref{thm:main1}.

\begin{proof}[Proof of Theorem \ref{thm:main1}]
By Proposition \ref{prop:key-qhs}, there exists a chain of two-sided ideals of $R$
\begin{align}
R= R\varepsilon_{0}R \supseteq \cdots \supseteq R\varepsilon_{l}R=ReR \supseteq R\varepsilon_{l+1}R \supseteq \cdots \supseteq R\varepsilon_{m}R=0 \notag
\end{align}
such that for each $i\in [0,m-1]$, $R\varepsilon_{i}R/R\varepsilon_{i+1}R$ is a heredity ideal of $R/R\varepsilon_{i+1}R$.
Hence $R$ is quasi-hereditary.
\end{proof}

If we assume only $ReR\in \mathsf{Proj}R$, then Theorem \ref{thm:main1} does not necessarily hold true.

\begin{example}
Let $K$ be a field and let $A$ be the bound quiver algebra defined by the quiver
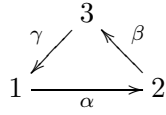
\begin{align} 
\xymatrix@=15pt{&3 \ar[ld]_{\gamma} \\
1 \ar[rr]_{\alpha}  && 2\ar[lu]_{\beta} \\
}\notag
\end{align}
with relations $\alpha\beta\gamma\alpha$ and $\gamma\alpha\beta$ (for the definition of bound quiver algebras, see \cite{ASS06}). 
Let $e:=e_{1}+e_{3}$, where $e_{i}$ is the primitive idempotent corresponding to a vertex $i$. 
Then $A/AeA$ and $eAe$ are quasi-hereditary, $AeA\in\mathsf{Proj}A$ and $AeA\notin \mathsf{Proj}A^{\mathrm{op}}$.  
However, $A$ is not quasi-hereditary since there does not exist a heredity ideal of $A$. 
\end{example}

\section{Morita context rings and quasi-hereditary rings}

Using Theorem \ref{thm:main1}, we give a sufficient condition for Morita context rings to be quasi-hereditary.
Let $\mathcal{M}:=(R,S,M,N,\varphi,\psi)$ be a Morita context, that is, $R,S$ are associative unital rings, $M$ is an $R$-$S$-bimodule, $N$ is an $S$-$R$-bimodule, $\varphi:M\otimes_{S}N\to R$ is an $R$-$R$-bimodule homomorphism and $\psi:N\otimes_{R}M\to S$ is an $S$-$S$-bimodule homomorphism satisfying $\varphi(m\otimes n)m'=m\psi(n\otimes m')$ and $\psi(n\otimes m)n'=n\varphi(m\otimes n')$ for all $m,m'\in M$ and $n,n'\in N$. 
We define an associative unital ring $\Lambda(\mathcal{M})$, called a \emph{Morita context ring}, as 
\begin{align}
\Lambda(\mathcal{M}):=\left( \begin{array}{cc}
R&M\\
N&S
\end{array} \right):=
\left\{ \left( \begin{array}{cc}
r&m\\
n&s
\end{array} \right)\;\middle|\; r\in R,\ m\in M,\ n\in N, s\in S \right\}\notag
\end{align}
with multiplication 
\begin{align}\left( \begin{array}{cc}
r&m\\
n&s
\end{array} \right)\left( \begin{array}{cc}
r'&m'\\
n'&s'
\end{array} \right)=\left( \begin{array}{cc}
rr'+\varphi(m\otimes n')&rm'+ms'\\
nr'+sn'&\psi(n\otimes m')+ss'
\end{array} \right).\notag
\end{align}
Our aim of this paper is to prove the following result.

\begin{theorem}\label{thm:main2}
Let $\Lambda(\mathcal{M})$ be a Morita context ring which is semiprimary.
Assume that the following four conditions are satisfied.
\begin{itemize}
\item[(a)] $R/\operatorname{Im}\varphi$ and $S$ are quasi-hereditary rings.
\item[(b)] $M$ is a projective $S$-module.
\item[(c)] $N$ is a projective $S^{\mathrm{op}}$-module.
\item[(d)] $\varphi$ is a monomorphism.
\end{itemize}
Then $\Lambda(\mathcal{M})$ is a quasi-hereditary ring.
\end{theorem}

\begin{proof}
Let $\Lambda:=\Lambda(\mathcal{M})$ and $e:=\left( \begin{array}{cc}
0&0\\
0&1_{S}
\end{array} \right)\in \Lambda$.
Then $e$ is an idempotent of $\Lambda$ satisfying $\Lambda/\Lambda e\Lambda\cong R/\operatorname{Im}\varphi$ and $e\Lambda e\cong S$.
By (a), $\Lambda/\Lambda e\Lambda$ and $e\Lambda e$ are quasi-hereditary.
Moreover, $e\Lambda e\Lambda=e\Lambda\in \mathsf{Proj}\Lambda$ and $\Lambda e\Lambda e=\Lambda e\in\mathsf{Proj}\Lambda^{\mathrm{op}}$.
By Theorem \ref{thm:main1}, it is enough to show $(1-e)\Lambda e\Lambda \in \mathsf{Proj}\Lambda$ and $\Lambda e\Lambda (1-e)\in \mathsf{Proj}\Lambda^{\mathrm{op}}$.
By \cite[(1.5) Theorem]{G82}, we obtain isomorphisms
\begin{align}
(1-e)\Lambda e \Lambda=
\left( \begin{array}{cc}
\operatorname{Im}\varphi&M\\
0&0
\end{array} \right)\cong 
\left( \begin{array}{cc}
\operatorname{Im}\varphi&M
\end{array} \right)\cong 
\left( \begin{array}{cc}
M\otimes_{S}N&M
\end{array} \right)\notag
\end{align}
as a $\Lambda$-module, where the last isomorphism follows from (d).
On the other hand, we have isomorphisms
\begin{align}
e\Lambda=\left( \begin{array}{cc}
0&0\\
N&S
\end{array} \right)\cong \left( \begin{array}{cc}
N&S
\end{array} \right)\cong \left( \begin{array}{cc}
S\otimes_{S}N&S
\end{array} \right)\notag
\end{align}
as a $\Lambda$-module.
Thus it follows from (b) that $(1-e)\Lambda e\Lambda$ is a projective $\Lambda$-module.
This implies $\Lambda e\Lambda\in \mathsf{Proj}\Lambda$.
Similarly, we can show $\Lambda e\Lambda\in \mathsf{Proj}\Lambda^{\mathrm{op}}$ by (c) and (d).
Thus $\Lambda e\Lambda\in \mathsf{Proj}\Lambda \cap \mathsf{Proj}\Lambda^{\mathrm{op}}$.
Hence $\Lambda$ is quasi-hereditary.
\end{proof}

By Theorem \ref{thm:main2}, we show that an upper triangular matrix ring of a quasi-hereditary ring is also quasi-hereditary.

\begin{example}\label{ex:utm}
Let $R$ be a quasi-hereditary ring and let
\begin{align}
\Lambda:=\left( \begin{array}{cc}
R'&M\\
0&R
\end{array} \right):=\left( \begin{array}{cccc|c}
R & R & \cdots & R & R \\
  & R & \cdots & R & R \\
  &   & \ddots & \vdots    & \vdots  \\ 
 \text{\LARGE{0}} &   &        & R & R \\\hline
0 & 0 & \cdots & 0 & R
\end{array} \right)\notag
\end{align}
be an upper triangular matrix ring.
Then $\mathcal{M}:=(R', R, M, 0, 0, 0)$ is clearly a Morita context and $\Lambda=\Lambda(\mathcal{M})$.
Since $R$ is a semiprimary ring, the upper triangular matrix ring $\Lambda$ is also semiprimary.
Moreover, $M$ is a direct summand of a direct sum of a copy of $R$ as an $R$-module.
Thus $M$ is a projective $R$-module.
By induction on the size of the matrix, $R'$ is also quasi-hereditary.
Hence it follows from Theorem \ref{thm:main2} that $\Lambda$ is quasi-hereditary.
\end{example}

\section{Block extensions of quasi-hereditary rings}

In this section, we study block extensions of a quasi-hereditary ring.
We start with recalling the definition of block extensions of semiperfect rings. For details, refer to \cite{O90, BO09}.
Let $R$ be a basic semiperfect ring. 
Then there exists a complete set $\{ e_{1}, e_{2}, \ldots, e_{n}\}$ of primitive orthogonal idempotents.
For each $i,j\in [1, n]$, we put $R_{ij}:=e_{i}Re_{j}$ and $R_{i}:=R_{ii}$. Note that $R_{i}$ is a subring of $R$ with identity $e_{i}$ and $R_{ij}$ is an $R_{i}$-$R_{j}$-bimodule.
Let $l_{1}, l_{2}, \ldots, l_{n}$ be positive integers.
For each $i\in [1,n]$, let $l_{\leq i}:=l_{1}+l_{2}+\cdots+l_{i}$, $l:=l_{\leq n}$ and $l_{\leq 0}:=0$.
For $i,s\in [1,n]$, we define an $l_{i}\times l_{s}$ matrix $B(i,s):=(b_{j,t})_{1\leq j\leq l_{i}, 1\leq t\leq l_{s}}$, 
where the $(j,t)$-entry $b_{j,t}$ is given by
\begin{align}
b_{j,t}=
\begin{cases}
\ R_{i} &(i=s,\ j\leq t),\\
\ J(R_{i}) &(i=s,\ j>t),\\
\ R_{is} &(i\neq s).
\end{cases}\notag
\end{align}
Note that $B(i,i)$ is a basic ring with usual matrix multiplication and $B(i,s)$ is a $B(i,i)$-$B(s,s)$-bimodule.

\begin{definition}\label{def:block-ext}
We define the \emph{block extension} $B:=B(R; l_{1},l_{2},\ldots, l_{n})$ of $R$ for $(l_{1}, l_{2},\ldots, l_{n})$ as an $l\times l$ matrix with block decomposition
\begin{align}
\left( \begin{array}{cccc}
B(1,1)&B(1,2)&\cdots &B(1,n)\\
B(2,1)&B(2,2)&\cdots &B(2,n)\\
\vdots&\vdots&\ddots&\vdots\\
B(n,1)&B(n,2)&\cdots&B(n,n)
\end{array} \right).\notag
\end{align}
Clearly, $B$ forms a basic ring with usual matrix multiplication. 
\end{definition}

Note that each block extension clearly forms a Morita context as
\begin{align}
B=\left(
\begin{array}{ccc|ccc}
B(1,1)&\cdots&B(1,k-1)&B(1,k)&\cdots&B(1,n)\\
\vdots&\ddots&\vdots  &\vdots&      &\vdots\\
B(k-1,1)&\cdots&B(k-1,k-1)&B(k-1,k)&\cdots&B(k-1,n)\\\hline
B(k,1)&\cdots&B(k,k-1)&B(k,k)&\cdots&B(k,n)\\
\vdots&      &\vdots  &\vdots&\ddots&\vdots\\
B(n,1)&\cdots&B(n,k-1)&B(n,k)&\cdots&B(n,n)
\end{array}
\right)=:
\left(
\begin{array}{cc}
B^{(k)}_{11}&B^{(k)}_{12}\\
B^{(k)}_{21}&B^{(k)}_{22}
\end{array}
\right)\notag
\end{align}
and multiplications $\varphi^{(k)}:B^{(k)}_{12}\otimes_{B^{(k)}_{22}}B^{(k)}_{21}\to B^{(k)}_{11}$, $\psi^{(k)}:B^{(k)}_{21}\otimes_{B^{(k)}_{11}}B^{(k)}_{12}\to B^{(k)}_{22}$. 
Namely, $\mathcal{M}:=(B^{(k)}_{11},B^{(k)}_{22},B^{(k)}_{12},B^{(k)}_{21},\varphi^{(k)},\psi^{(k)})$ is a  Morita context and $B=\Lambda(\mathcal{M})$. 
By $\operatorname{Im}\varphi^{(k)} \subset J(B^{(k)}_{11})$ and $\operatorname{Im}\psi^{(k)}\subset J(B^{(k)}_{22})$, it follows from \cite[Theorem 2.7(5)]{TLZ14} that $B$ is a semiprimary ring if $R$ is semiprimary.

By regarding block extensions as Morita context rings, it follows from Theorem \ref{thm:main2} that quasi-hereditary rings are closed under taking block extensions.

\begin{theorem}\label{thm:main3}
Let $R$ be a basic semiprimary ring with complete set $\{ e_{1}, e_{2}, \ldots, e_{n}\}$ of primitive orthogonal idempotents of $R$ and let $\ell_{1}, \ell_{2}, \ldots, \ell_{n}$ be positive integers. 
If $R$ is quasi-hereditary, then so is the block extension $B(R;\ell_{1}, \ell_{2},\ldots,\ell_{n})$.
\end{theorem}

In the following, we assume that $R$ is a semiprimary ring and $\{ e_{1}, e_{2}, \ldots, e_{n}\}$ is a complete set of primitive orthogonal idempotents.
Let $B:=B(R;l_{1},l_{2},\ldots, l_{n})$ be the block extension.
For each $i\in [1,n]$, we define $\tilde{e}_{i}\in B$ as $\tilde{e}_{i}:=(b_{j,t})$, where 
\begin{align}
b_{j,t}=
\begin{cases}
e_{i} &(j=t\in [l_{\leq i-1}+1, l_{\leq i}]),\\
0 &(\textnormal{otherwise}).
\end{cases}\notag
\end{align}
Then $\tilde{e}_{i}$ is an idempotent of $B$ and $\tilde{e}_{i}B\tilde{e}_{i}\cong B(i,i)= B(R_{i};l_{i})$.
Moreover, we have the following result.

\begin{lemma}\label{lem:fac-ring}
For each $i\in[1,n]$, the factor ring $B/B\tilde{e}_{i}B$ is isomorphic to the block extension $B(R/Re_{i}R;l_{1},\ldots, l_{i-1},l_{i+1},\ldots,l_{n})$.
\end{lemma}

\begin{proof}
For simplicity, assume $i=n$. Then we have
\begin{align}\label{eq:matrix}
B\tilde{e}_{n}=\left( \begin{array}{cc}
&B(1,n)\\
\text{\LARGE{0}}&\vdots\\
&B(n-1,n)\\
&B(n,n)
\end{array} \right),\ 
\tilde{e}_{n}B=\left( \begin{array}{cccc}
&\text{\LARGE{0}}&&\\
B(n,1)&\cdots&B(n,n-1)&B(n,n)
\end{array} \right).
\end{align}
We show $B/B\tilde{e}_{n}B\cong B(R/Re_{n}R;l_{1},\ldots, l_{n-1})$.
By \eqref{eq:matrix}, we obtain 
\begin{align}
B\tilde{e}_{n}B=
\left(
\begin{array}{cccc}
B(1,n)B(n,1)&\cdots&B(1,n)B(n,n-1)&B(1,n)\\
\vdots&\ddots&\vdots&\vdots\\
B(n-1,n)B(n,1)&\cdots&B(n-1,n)B(n,n-1)&B(n-1,n)\\
B(n,1)&\cdots&B(n,n-1) &B(n,n)
\end{array} 
\right)\notag
\end{align}
and
\begin{align}
B/B\tilde{e}_{n}B\cong
\left(
\begin{array}{ccc}
B(1,1)/B(1,n)B(n,1)&\cdots&B(1,n-1)/B(1,n)B(n,n-1)\\
\vdots&\ddots&\vdots\\
B(n-1,1)/B(n-1,n)B(n,1)&\cdots&B(n-1,n-1)/B(n-1,n)B(n,n-1)
\end{array} 
\right).\notag
\end{align}
Let $\overline{(-)}:R\to R/Re_{n}R$ be a natural surjection.
If $i\neq s$, then each entry of $B(i,s)/B(i,n)B(n,s)$ is 
\begin{align}
e_{i}Re_{s}/e_{i}Re_{n}Re_{s}\cong \overline{e}_{i}\overline{R}\overline{e}_{s}.\notag
\end{align}
Assume $i=s$. If $j\leq t$, then the $(j,t)$-entry of $B(i,i)/B(i,n)B(n,i)$ is
\begin{align}
e_{i}Re_{i}/e_{i}Re_{n}Re_{i}\cong \overline{e}_{i}\overline{R}\overline{e}_{i}.\notag
\end{align}
If $j>t$, then the $(j,t)$-entry of $B(i,i)/B(i,n)B(n,i)$ is 
\begin{align}
e_{i}J(R)e_{i}/e_{i}Re_{n}Re_{i}=J(e_{i}Re_{i})/e_{i}Re_{n}Re_{i}\overset{\text{($\ast$)}}{=}J(e_{i}Re_{i}/e_{i}Re_{n}Re_{i})\cong \overline{e}_{i}J(\overline{R})\overline{e}_{i},\notag
\end{align}
where ($\ast$) follows from $e_{i}Re_{n}Re_{i}\subseteq e_{i}J(R)e_{n}J(R)e_{i}\subseteq J(e_{i}Re_{i})$.
Thus $B/B\tilde{e}_{n}B\cong B(\overline{R};l_{1},\ldots, l_{n-1})$. This completes the proof.
\end{proof}

Now, we are ready to prove Theorem \ref{thm:main3}.

\begin{proof}[Proof of Theorem \ref{thm:main3}]
Let $R$ be a quasi-hereditary ring. 
By Proposition \ref{prop:DR2-710}, there exists an idempotent $e\in R$ such that $ReR\in \mathsf{Proj}R\cap\mathsf{Proj}R^{\mathrm{op}}$ and $eJ(R)e=0$. By \cite[Proposition 1.3]{UY90}, we may assume $e=e_{n}$.

Recall that $\mathcal{M}:=(B_{11}^{(n)},B_{22}^{(n)},B_{12}^{(n)},B_{21}^{(n)},\varphi^{(n)},\psi^{(n)})$ is a Morita context and $B=\Lambda(\mathcal{M})$. For brevity, we omit the index $n$ in the Morita context.
By Theorem \ref{thm:main2}, it is enough to show that (a) $B_{11}/\operatorname{Im} \varphi$ and $B_{22}$ are quasi-hereditary, (b) $B_{12}\in \mathsf{Proj}B_{22}$, (c) $B_{21}\in \mathsf{Proj}B_{22}^{\mathrm{op}}$, and (d) $\varphi$ is a monomorphism. 

We show (b). We have an isomorphism $B_{12}\cong B(1,n)\oplus B(2,n)\oplus \cdots \oplus B(n-1,n)$ as a $B_{22}$-module.
By $Re_{n}R\in \mathsf{Proj}R$, it follows from Lemma \ref{lem:keyDRT}(2) that $(1-e_{n})Re_{n}$ is a projective $R_{n}$-module. For each $i\in[1,n-1]$, we obtain $R_{in}\in \mathsf{Add}R_{n}$. 
Thus we have $B(i,n)\cong \left( \begin{array}{cccc}
R_{in}&R_{in}&\cdots &R_{in}\\
\end{array} \right)^{\oplus l_{i}}\in \mathsf{Add}(e_{n,1}B_{22})\subset \mathsf{Proj}B_{22}$, where $e_{n,1}=(b_{j,t})\in B_{22}=B(n,n)$ is given by $b_{j,t}=e_{n}$ if $j=t=1$ and $0$ if otherwise.
Hence $B_{12}\in\mathsf{Proj}B_{22}$.
Similarly, we can show (c).
We show (d). By the universality of tensor products, we have isomorphisms
\begin{align}
B_{12}\otimes_{B_{22}}B_{21}
&=\left( \begin{array}{c}
B(1,n)\\ \vdots\\ B(n-1,n)
\end{array} \right) \otimes_{B_{22}} \left( \begin{array}{ccc}
B(n,1)& \cdots& B(n,n-1)
\end{array} \right)\notag\\
&\cong \left( \begin{array}{ccc}
B(1,n)\otimes_{B(n,n)}B(n,1)& \cdots& B(1,n)\otimes_{B(n,n)}B(n,n-1)\\
\vdots&\ddots&\vdots\\
B(n-1,n)\otimes_{B(n,n)}B(n,1)& \cdots& B(n-1,n)\otimes_{B(n,n)}B(n,n-1)
\end{array} \right)\notag
\end{align}
and
\begin{align}
B(i,n)\otimes_{B(n,n)} B(n,j)\cong \left( \begin{array}{ccc}
R_{in}\otimes_{R_{n}}R_{nj}& \cdots& R_{in}\otimes_{R_{n}}R_{nj}\\
\vdots&&\vdots\\
R_{in}\otimes_{R_{n}}R_{nj}& \cdots& R_{in}\otimes_{R_{n}}R_{nj}
\end{array} \right).\notag
\end{align}
By Lemma \ref{lem:keyDRT}(3), we obtain $R_{in}\otimes_{R_{n}}R_{nj}\cong e_{i}Re_{n}Re_{j}$, and hence $B(i,n)\otimes_{B(n,n)}B(n,j)\cong B(i,n)B(n,j)$.
Therefore $B_{12}\otimes_{B_{22}}B_{21}\cong B_{12}B_{21}$.
This implies that $\varphi$ is a monomorphism.
We show (a). 
Since $R_{n}$ is quasi-hereditary and $J(R_{n})=e_{n}J(R)e_{n}=0$, it follows from Example \ref{ex:utm} that $B_{22}=B(n,n)$ is a quasi-hereditary ring.
By Lemma \ref{lem:fac-ring}, we have $B/\operatorname{Im}\varphi \cong B/B\tilde{e}_{n}B\cong B(R/Re_{n}R; l_{1},l_{2},\ldots, l_{n-1})$. Since $R/Re_{n}R$ is a quasi-hereditary, $B(R/Re_{n}R; l_{1},l_{2},\ldots, l_{n-1})$ is inductively quasi-hereditary. 
Thus, by Theorem \ref{thm:main2}, the block extension $B$ is quasi-hereditary. 
\end{proof}

\subsection*{Acknowledgements}

The second author would like to express her deep gratitude to Kota Yamaura for helpful comments.
The authors are grateful to Ryoichi Kase for useful discussions.
The second author is supported by JSPS KAKENHI Grant Number JP23K12959.

\end{document}